\documentclass[]{interact}

\usepackage{epstopdf}
\usepackage[caption=false]{subfig}

\usepackage[numbers,sort&compress]{natbib}
\bibpunct[, ]{[}{]}{,}{n}{,}{,}

\theoremstyle{definition}

\theoremstyle{remark}

\numberwithin{equation}{section}
\usepackage{graphicx}
\usepackage{latexsym}
\usepackage{amsmath}
\usepackage{amssymb}
\usepackage{amsfonts}
\usepackage{verbatim}
\usepackage{mathrsfs}
\usepackage[modulo]{lineno} 
\usepackage[latin1]{inputenc}
\usepackage{color}
\usepackage[colorlinks,citecolor=blue,urlcolor=blue]{hyperref}

\newtheorem{tm}{Theorem}[section]
\newtheorem{rk}{Remark}[section]
\newtheorem{ap}{Assumption}[section]

\newtheorem{lm}{Lemma}[section]
\newtheorem{cor}{Corollary}[section]
\newtheorem{ex}{Example}[section]

\newcommand{\ee}{\mathbb E}
\newcommand{\pp}{\mathbb P}
\newcommand{\nn}{\mathbb N}
\newcommand{\rr}{\mathbb R}

\newcommand{\br}{\mathbf r}
\newcommand{\bs}{\mathbf s} 
\newcommand{\Bi}{\mathbf i}

\newcommand{\BB}{\mathcal B}
\newcommand{\CC}{\mathcal C}
\newcommand{\LL}{\mathcal L}

\newcommand{\OOO}{\mathscr O}
\newcommand{\FFF}{\mathscr F}

\newcommand{\<}{\langle}
\renewcommand{\>}{\rangle}
\allowdisplaybreaks \allowdisplaybreaks[4]

\begin{document}

\articletype{}

\title{Harnack Inequalities and Ergodicity of Stochastic Reaction-Diffusion Equation in $L^p$}

\author{
\name{Zhihui Liu \thanks{Email: liuzh3@sustech.edu.cn} 
}
\affil{Department of Mathematics \& National Center for Applied Mathematics Shenzhen (NCAMS) \& Shenzhen International Center for Mathematics, Southern University of Science and Technology, Shenzhen 518055, China}
}

\maketitle

\begin{abstract}
We derive Harnack inequalities for a stochastic reaction-diffusion equation with dissipative drift driven by additive irregular noise in the $L^p$-space for any $p \ge 2$. These inequalities are utilized to investigate the ergodicity of the corresponding Markov semigroup $(P_t)$. The main ingredient of our method is a coupling by the change of measure. Applying our results to the stochastic reaction-diffusion equation with a super-linear growth drift having a negative leading coefficient, perturbed by a Lipschitz term, indicates that $(P_t)$ possesses a unique and thus ergodic invariant measure in $L^p$ for all $p \ge 2$, which is independent of the Lipschitz term. 
\end{abstract}

\amscodename{: Primary 60H15; 37H05.}


\section{Introduction}

We consider the following stochastic reaction-diffusion equation under the homogeneous Dirichlet boundary condition on the bounded, open subset $\OOO$ of $\rr^d$:
\begin{align}\label{rd}
\frac{\partial X_t(\xi)}{\partial t}
=\Delta X_t(\xi) +f(X_t(\xi)) + G \frac{\partial W_t(\xi)}{\partial t},  
\quad (t, \xi) \in \rr_+ \times \OOO.
\end{align}  
Here $f$ has polynomial growth and satisfies a dissipativity condition, $(W_t)_{t \ge 0}$ is a cylindrical Wiener process, and $G$ is a densely defined closed linear operator which could be unbounded (see Section \ref{sec2} for more details).

The stochastic reaction-diffusion equation \eqref{rd} has numerous applications in material sciences and chemical kinetics \cite{ET89}. 
When $f(\xi)=\xi-\xi^3$, $\xi \in \rr$, Eq. \eqref{rd} is also called the stochastic Allen--Cahn equation or the stochastic Ginzburg--Landau equation.
It is widely used in many fields, for example, the random interface models and stochastic mean curvature flow \cite{Fun16}.
The existence of the invariant measure and ergodicity and their numerical correspondences for Eq. \eqref{rd} have been investigated in Hilbert settings, see, e.g., \cite{BS20, Hai02(PTRF), Kaw05(PA), LQ20(IMA), LQ21(SPDE), Liu22, LL24, Liu25, LS25} and references therein.

In contrast to SPDEs in Hilbert spaces, only a few papers treat the invariant measures and ergodicity for SPDEs, even with Lipschitz coefficients in Banach spaces.
The authors in \cite{BLS10(JEE)} studied invariant measures for SPDEs in martingale-type (M-type) 2 Banach spaces under Lipschitz and dissipativity conditions driven by regular noise.
For white-noise driven stochastic heat equation (Eq. \eqref{see} with Lipschitz coefficients), \cite{BR16(DCDS)} showed the uniqueness of the invariant measure, if it exists, on $L^p(0, 1)$ with $p>4$; the case for $p \in (2,4]$ remains unknown.
Recently, their method was extended in \cite{BK18(FS)} to an SPDE, arising in stochastic finance, in a weighted $ L^p$ space. 
See also \cite{Nee01(OTAA)} for the uniqueness of the invariant measure, if it exists, of the Ornstein--Uhlenbeck process (with form \eqref{df-wa}) using a pure analytical method. 

For SPDEs with non-Lipschitz coefficients, \cite{BG99(SPA)} obtained the existence of an invariant measure for Eq. \eqref{rd} in the space of continuous functions under the martingale solution framework; the uniqueness of the invariant measure was derived in \cite{Cer03(PTRF), Cer05(PTRF)} by taking advantage of the fact that a polynomial is uniformly continuous on bounded subsets of continuous functions. 
Recently, \cite{KN13(SPDE)} showed the existence of a unique invariant measure on the space of continuous complex functions for the stochastic complex Ginzburg--Landau (Eq. \eqref{rd} with $f(u)=- \Bi |u|^2 u$, where $\Bi=\sqrt{-1}$), relying on some reasonable estimates of the solution in the Hilbert--Sobolev spaces $\dot H^\beta$ with $\beta>d/2$, so that the noise is spatially regular enough.

To show the uniqueness of the invariant measure, these authors mainly formulated a Bismut--Elworthy--Li formula for the derivative of its Markov semigroup to get a gradient estimate, which shows the strong Feller property.
Then, the uniqueness of the invariant measure follows immediately from Khas'minskii--Doob theorem, provided an irreducibility condition holds.
The difficulties for the study of the uniqueness of an invariant measure for SPDEs in Banach settings arise mainly because the tools frequently used in the Hilbert space framework cannot be extended in a straightforward way to Banach space settings \cite{BR16(DCDS)}.

In the past decade, Wang-type dimension-free inequalities have been a new and efficient tool to study diffusion semigroups. 
They were first introduced in \cite{Wan97(PTRF)} for elliptic diffusion semigroups on non-compact Riemannian manifolds and in \cite{Wan10(JMPA)} for heat semigroups on manifolds with boundary.
Roughly speaking, such inequality for a Markov semigroup $(P_t)$ in a Banach space $E$ is formulated as  
\begin{align} \label{har}
\Phi(P_t \phi(x)) & \le P_t (\Phi(\phi)(y)) \exp \Psi(t, x, y),
\quad t>0, \ x, y \in E, \ \phi \in \BB^+_b(E),
\end{align}
where $\Phi: [0, \infty) \rightarrow [0, \infty)$ is convex, $\Psi$ is nonnegative on $[0, \infty) \times E \times E$ with $\Psi(t, x, x)=0$ for all $t>0$ and $x \in E$, and $\BB^+_b(E)$ denotes the family of all Borel measurable and bounded, nonnegative functions on $E$.

There are two frequently used choices of $\Phi$.
One is given by a power function $\Phi(\xi)=\xi^\bs$, $\xi \ge 0$, for some $\bs>1$, where \eqref{har} reduces to 
\begin{align} \label{har-1}
(P_t \phi(x))^\bs & \le P_t \phi^\bs(y)  \exp \Psi(t, x, y),
\quad t>0, \ x, y \in E, \ \phi \in \BB^+_b(E).
\end{align}
Another is given by $\Phi(\xi)=e^\xi$, $\xi \in \rr$, in which one may use $\log \phi$ to replace $\phi$, so that \eqref{har} becomes 
\begin{align} \label{har-2}
P_t \log \phi(x) & \le \log P_t \phi(y)+\Psi(t, x, y),
\quad t>0, \ x, y \in E, \ \phi \in \BB^+_b(E).
\end{align}

The above two inequalities \eqref{har-1}-\eqref{har-2} are called the power-Harnack inequality and the log-Harnack inequality, respectively.
Both inequalities have been investigated extensively and applied to SODEs and SPDEs via coupling by the change of measure, see, e.g., \cite{Kaw05(PA), Wan07(AOP), WZ13(JMPA), WZ14(SPA)}, the monograph \cite{Wan13}, and references therein. 
Besides the gradient estimate, which yields the strong Feller property, these Harnack inequalities also have a lot of other applications.
For example, they are used to study the contractivity of the Markov semigroup in \cite{Wan11(AOP), Wan17(JFA)} and to derive almost surely (a.s.) strictly positivity of the solution for an SPDE in \cite{Wan18(PTRF)}.

For SPDEs with polynomial growth drift driven by irregular noise in Hilbert settings, we are only aware of \cite{Kaw05(PA), Xie19(JDE)} investigating Harnack inequalities in a weighted $L^2$-space and the nonnegative subset of $L^2$, respectively.
When the noise is of trace-class, i.e., $G$ appearing in Eq. \eqref{rd} is a Hilbert--Schmidt operator, then the variational solution theory can be used and multiplicative noise can also be considered, as the solution is a semi-martingale so that It\^o formula can be applied; see, e.g., \cite{HLL20(SPL), Liu09(JEE)}. 
On the contrary, the variational solution would not exist in the irregular noise case; one needs to adopt the mild solution theory instead, where It\^o formula is unavailable.  
We also note that to derive Harnack inequalities for white noise-driven SPDE, \cite{WZ14(SPA)} used finite-dimensional approximations to get a sequence of SODEs such that the arguments developed in \cite{Wan11(AOP)} can be applied.

The previous questions motivate the study investigating Harnack inequalities and ergodicity of Eq. \eqref{rd} in the Banach space $L^p:=L^p(\OOO)$.
To derive the above two types of Harnack inequalities for Eq. \eqref{rd} on $(L^p)_{p \ge 2}$-spaces, our main idea is the construction of a coupling by the change of measure and a uniform pathwise estimate for this coupling.
As by-products, a gradient estimate and the uniqueness of the invariant measure, if it exists, follow immediately.
To our knowledge, the obtained Harnack inequalities (see \eqref{har-log}-\eqref{har-pow}) are the first two Harnack inequalities for SPDEs in Banach settings.

On the other hand, to show the existence of an invariant measure for Eq. \eqref{rd} with super-linear growth but without strong dissipativity drift, we establish the tightness of the sequence of empirical measures $(\mu_n)$ (defined in \eqref{mun}) via a compact embedding to $L^p$ through the Sobolev--Slobodeckii space $(W^{\beta, p}, \|\cdot\|_{\beta, p})$. 
This is acheived by deriving a uniform estimate of $\mu_n(\|\cdot\|_{\beta, p})$, where a uniform estimate of $\mu_n(\|\cdot\|_{q+p-2}^{q+p-2})$ plays a key role.  
In combination with the uniqueness result, we obtain the existence of a unique and thus ergodic invariant measure for Eq. \eqref{rd} (in Theorem \ref{main2}).

The rest of the paper is organized as follows.  
Some preliminaries, assumptions, and main results are given in the next section. 
We derive a uniform pathwise estimate to get the existence of a unique global solution to Eq. \eqref{rd} in the first part of Section \ref{sec3}.
In another part of Section \ref{sec3}, we construct the coupling and derive a uniform pathwise estimation for this coupling.
These estimations ensure the well-posedness of the coupling and will be used in the last section to derive Harnack inequalities and to prove the main results in Section \ref{sec2}.

\section{Preliminaries and Main Results}
\label{sec2}

Let $\OOO \subset \rr^d$, $d \ge 1$, be an open, bounded Lipschitz domain. 
Throughout, $p \ge 2$ is a fixed constant.
Denote by $(L^p=L^p(\OOO), \|\cdot\|_p)$ and $H_0^1=H_0^1(\OOO)$ the usual Lebesgue and Sobolev spaces on $\OOO$, respectively. Let $\BB_b(L^p)$ be the class of bounded Borel measurable functions on $L^p$, and $\BB^+_b(L^p)$ and $\CC_b(L^p)$ the nonnegative and continuous subsets in $\BB_b(L^p)$, respectively.  
In particular, when $p=2$, $L^2$ is a Hilbert space with the norm $\|\cdot\|:=\|\cdot\|_2$ and the inner product $(\cdot, \cdot)$.

Let $(W_t)_{t \ge 0}$ be a $U:=L^2$-valued cylindrical Wiener process concerning a complete filtered probability space $(\Omega, \FFF, (\FFF_t)_{t \ge 0},\pp)$ satisfying the usual condition, 
i.e., there exists an orthonormal basis $(e_k)_{k=1}^\infty$ of $L^2$ and a family of independent standard real-valued Brownian motions $(\beta_k)_{k=1}^\infty $ such that  
\begin{align} \label{df-W}
W_t=\sum_{k=1}^\infty e_k\beta_k(t),\quad t \ge 0.
\end{align}

Denote by $A$ the Dirichlet Laplacian operator on $L^p$.
It is well-known that $A$ generates an analytic $\CC_0$-semigroup in $L^p$, denoted by $(S^p_t)_{t \ge 0}$, for each $p \ge 2$. 
These semigroups are consistent, in the sense that $S^{p_1}_t x=S^{p_2}_t x$, for all $t \ge 0$, $x \in L^{p_1} \cap L^{p_2}$, and $p_1, p_2 \ge 2$.
Then we shall denote all $(S^p_t)_{t \ge 0}$ by $(S_t)_{t \ge 0}$ for all $p \ge 2$, if there is no confusion.  
Moreover, $(S_t)_{t \ge 0}$ satisfies the following ultracontractivity (see, e.g., \cite[Section 2.1]{Cer03(PTRF)}):
\begin{align} \label{ult}
\|S_t u\|_{\beta, \br} \le C e^{-\lambda_1 t} t^{-(\frac\beta2+\frac{d(\br-\bs)}{2\br \bs})} \|u\|_{\bs}, 
\quad t>0, \ u \in L^\bs,
\end{align} 
for all $\beta \in (0,1)$ and $1 \le \bs \le \br \le \infty$, where $\lambda_1>0$ is the first eigenvalue of $-A$.
For convenience, here and what follows, we frequently use the generic constant $C$, which may differ in each appearance. 
When $p=2$, the following Poincar\'e inequality holds:
\begin{align} \label{poin}
\|\nabla u\|_2 \ge \lambda_1 \|u\|_2, \quad u \in H_0^1.
\end{align}  
In Section \ref{sec4}, we also need the Sobolev--Slobodeckij space $W^{\beta,p}$, with $\beta\in (0,1)$, whose norm is defined by
\begin{align} \label{sob-slo}
\|\phi\|_{\beta,p}
:=\bigg(\|\phi\|_p^p 
+\int_{\OOO}\int_{\OOO} 
\frac{|\phi(\xi)-\phi(\eta)|^p}{|\xi-\eta|^{d+\beta p}}
{\rm d}\xi {\rm d}\eta \bigg)^\frac1p.
\end{align} 
It is known that the following compact embedding holds (see, e.g., \cite[Theorem 7.1]{NPV12}):
\begin{align} \label{emb}
W^{\beta, p} \subset L^p, \quad \beta \in (0, d/p).
\end{align}

\subsection{Main assumptions and results}

Let us give the following assumptions on the data of Eq. \eqref{rd}.
We begin with the conditions on the drift function $f$.

\begin{ap} \label{ap-f} 
There exist constants $L_f \in \rr$, $\theta, L'_f>0$, and $q \ge 2$ such that for all $\xi, \eta \in \rr$,
\begin{align}
& (f(\xi)-f(\eta)) (\xi-\eta) \le L_f |\xi-\eta|^2 - \theta |\xi-\eta|^q, \label{f} \\
& |f(\xi)-f(\eta)| \le L'_f (1+|\xi|^{q-2}+|\eta|^{q-2})|\xi-\eta|.  \label{f+}
\end{align} 
\end{ap}

\begin{rk} \label{rk-f}
A motivating example of $f$ such that Assumption \ref{ap-f} holds is a polynomial of odd order $q-1$ with a negative leading coefficient (for the stochastic Allen--Cahn equation, $q= 4$), perturbed with a Lipschitz continuous function; see, e.g., \cite[Example 7.8]{PZ14}.
\end{rk}

Define by $F$ the Nemytskii operator associated with $f$, i.e.,  
\begin{align} \label{df-F}
F(x)(\xi)=f(x(\xi)), \quad x \in L^p, \ \xi \in \OOO.
\end{align} 
Then Eq. \eqref{rd} can be rewritten as the stochastic evolution equation 
\begin{align}\label{see} 
{\rm d}X_t &=(A X_t+F(X_t)) {\rm d}t+ G{\rm d}W_t, \quad t > 0, 
\end{align}
with the initial datum $X_0=x \in L^p$, where $F$ is the Nemytskii operator defined in \eqref{df-F} associated with $f$ and $W$ is a $U$-valued cylindrical Wiener process given in \eqref{df-W}.

\begin{rk} \label{rk-F}
It follows from \eqref{f}-\eqref{f+} that the Nemytskii operator $F$ defined in \eqref{df-F} is a well-defined, continuous operator from $L^q$ to $L^{q'}$, with $q'=q/(q-1)$.
Moreover,   
\begin{align*}
& \<F(u)-F(v), u-v\>
 \le L_f \|u-v\|^2-\theta \|u-v\|_q^q, 
 \quad \forall~u, v\in L^q,
\end{align*}  
where $\<\cdot, \cdot\>$ is the dualization between $L^{q'}$ and $L^q$ with respect to $L^2$.
\end{rk}

To perform the assumption of the noise part, as we consider the Banach spaces $(L^p)_{p \ge 2}$, let us first recall the required materials of stochastic calculus in Banach spaces, especially the M-type 2 spaces and the $\gamma$-radonifying operators.
It is known that the stochastic calculations in Banach spaces depend heavily on the geometric structure of the underlying spaces.

We first recall the definitions of the M-type for a Banach space.
A Banach space $E$ is called of M-type $2$ if there exists a constant $\tau^M \ge 1$ such that  
\begin{align*}
\|f_N\|_{L^2( \Omega'; E)}
\le \tau^M \Big(\|f_0\|_{L^2(\Omega';E)}^2
+\sum_{n=1}^N \|f_n-f_{n-1}\|_{L^2(\Omega';E)}^2 \Big)^{1/2}
\end{align*}
for all $E$-valued square integrable martingales $\{f_n\}_{n=0}^N$. 
For stochastic calculus in Banach spaces, the so-called $\gamma$-radonifying operators play important roles instead of Hilbert--Schmidt operators in Hilbert settings. 
Let $(\gamma_n)_{n \in \nn_+}$ be a sequence of independent 
$\mathcal N(0,1)$-random variables in a probability space $(\Omega',\FFF',\pp')$.
Denote by $\LL(U, E)$ the space of linear operators from $U$ to $E$.
An operator $R\in \LL(U, E)$ is called $\gamma$-radonifying if there exists an orthonormal basis $(h_n)_{n\in \nn_+}$ of $U$ such that the Gaussian series $\sum_{n\in \nn_+}\gamma_nRh_n$ converges in $L^2(\Omega';E)$.
In this situation, it is known that the number 
\begin{align*}
\|R\|_{\gamma(U, E)}:=\Big\|\sum_{n\in \nn_+}\gamma_nRh_n\Big\|_{L^2(\Omega';E)}
\end{align*}
does not depend on the sequence $(\gamma_n)_{n \in \nn_+}$ and the basis $(h_n)_{n\in \nn_+}$, and
it defines a norm on the space $\gamma(U, E)$ of all $\gamma$-radonifying operators from $U$ to $E$.
In particular, if $E$ reduces to a Hilbert space, then $\gamma(U, E)$ coincides with the space of all Hilbert--Schmidt operators from $U$ to $E$.

Let $T>0$ and $(E, \|\cdot\|_E)$ be an M-type 2 space.
For any $\gamma(U, E)$-valued adapted process $\Phi \in L^\bs(\Omega; L^2(0,T;\gamma(U, E)))$ with $\bs \ge 2$, the following one-sided Burkholder inequality for the $E$-valued stochastic integral $\int_{0}^t \Phi_r{\rm d}W_r$ holds for some constant $C=C(\bs)$ (see \cite[Theorem 2.4]{Brz97(SSR)} or \cite[(3)]{HHL19(JDE)}):
\begin{align}\label{bdg}
\ee \sup_{t\in [0,T]} \Big\|\int_{0}^t \Phi_r{\rm d}W_r\Big\|_E^\bs 
\le C \ee \Big(\int_0^T \|\Phi\|^2_{\gamma(U, E)} {\rm d} t\Big)^{\bs/2}.
\end{align}
Returning to our case, it is known that $(L^p)_{p \ge 2}$ are M-type 2 spaces. 
For more details about definitions and properties of M-type 2 spaces and 
$\gamma$-radonifying operators, we refer to \cite{Brz97(SSR)}.

With these preliminaries, we can now give our assumptions on $G$ appearing in Eq. \eqref{see}.
Let $\mathcal L_S(U)$ be the set of all densely defined closed linear operators $(L, {\rm Dom}(L))$ on $U$ such that for every $t>0$, $S_t L$ extends to a unique operator in $\gamma(U, L^p)$, which is again denoted by $S_t L$.
Assume that $G \in L_S(U)$ and $\int_0^T \|S_t G\|^2_{\gamma(U, L^p)} {\rm d}t<\infty$ for each fixed $T>0$.
Denote by $W_A$ the stochastic convolution, also known as the Ornstein--Uhlenbeck process, associated with Eq. \eqref{see}, i.e.,  
\begin{align} \label{df-wa} 
W_A(t)=\int_0^t S_{t-r} G {\rm d}W_r, \quad t \ge 0.
\end{align}
It is clear that $W_A$ is the mild solution of the linear equation
\begin{align} \label{eq-wa} 
{\rm d}Z_t &=A Z_t {\rm d}t+G {\rm d}W_t,
\quad Z_0=0,
\end{align}
i.e., Eq. \eqref{see} with $F=0$, with vanishing initial datum.
It follows from the Burkholder inequality \eqref{bdg} that 
\begin{align*}
\ee \|W_A(T)\|_p^2
\le C \int_0^T \|S_t G\|^2_{\gamma(U, L^p)} {\rm d}t <\infty,
\end{align*}
which shows that $W_A(T)$ possesses the bounded second moment in $L^p$. 
Moreover, we perform the following stronger assumption: we shall handle the polynomial drift function $f$ satisfying Assumption \ref{ap-f} where $q \ge 2$ was introduced.

\begin{ap} \label{ap-wa}
$W_A$ has a continuous version in $L^{q+p-2}$ such that 
\begin{align} \label{wa1}
\sup_{t \ge 0} \ee \|W_A(t)\|_{q+p-2}^{q+p-2} <\infty.
\end{align} 
\end{ap} 

The above Assumption \ref{ap-wa} on the Ornstein--Uhlenbeck process $W_A$ defined in \eqref{df-wa} is necessary for the study of the well-posedness, Harnack inequalities, and ergodicity in Sections \ref{sec3} and \ref{sec4}.
To indicate that Assumption \ref{ap-wa} is natural, we remark that the condition \eqref{wa1} is valid in various applications, even when $G$ is an unbounded operator.

\begin{ex} \label{ex-wa}
Consider 1D Eq. \eqref{eq-wa} with $\OOO=(0,1)$, $U=L^2(0,1)$ (with a uniformly bounded orthonormal basis $(e_k=\sqrt2 \sin (k\pi x))_{k \in \nn_+}$), and $G: =(-A)^{\theta/2}$ for some $\theta<1/2$.
In particular, $\theta=0$ corresponds to white noise and $G$ is an unbounded operator when $\theta \in (0, 1/2)$.

It is clear that $G$ is a densely defined closed linear operator on $U$.
To show $G \in L_S(U)$, we note that for $\br \ge 2$, $L^\br$ is a Banach function space with finite cotype, so an operator $\Phi\in \gamma(U, L^\br)$ if and only if $(\sum_{k=1}^\infty (\Phi e_k)^2)^{1/2}$ belongs to $L^\br$ and there exists a constant  $C>0$ such that (see \cite[Lemma 2.1]{NVW08(JFA)})
\begin{align} \label{gamma}
\frac1C \Big\|\sum_{k=1}^\infty (\Phi e_k)^2\Big\|_{\br/2}
\le \|\Phi\|^2_{\gamma(U, L^\br)}
&\le C \Big\|\sum_{k=1}^\infty (\Phi e_k)^2\Big\|_{\br/2},
\quad \Phi\in \gamma(U, L^\br).
\end{align} 
It follows from the above estimates with $\Phi=S_t G$ for $t>0$ that 
\begin{align*} 
\|S_t G\|^2_{\gamma(U, L^\br)}
\le C \Big\|\sum_{k=1}^\infty (S_t G e_k)^2\Big\|_{\br/2} 
\le C \sum_{k=1}^\infty e^{-2 \lambda_k t} \lambda_k^\theta,
\quad \text{with}~ \lambda_k=\pi^2 k^2, ~ k \in \nn_+,
\end{align*} 
which is convergent if and only if $\theta<1/2$.
This shows $S_t G \in \gamma(U, L^p)$ for every $t>0$ and thus $G \in L_S(U)$.

Finally, one can use the Burkholder inequality \eqref{bdg} and the above equivalence relation \eqref{gamma} to show that 
$W_A$ belongs to $\CC([0, \infty); L^{\bs}(\Omega; L^\br))$ for any $\bs, \br \ge 2$, following an argument used in \cite[(2.17) in Lemma 2.2]{LQ20(IMA)}.
Indeed, for any $t \ge 0$,  
\begin{align*} 
\sup_{t \ge 0} \ee \|W_A(t)\|_{\br}^{\bs}
& \le C(\bs) \sup_{t \ge 0} \Big(\int_0^t \|S_r (-A)^\frac\theta2\|^2_{\gamma(L^2, L^\br)} {\rm d}r\Big)^{\bs/2} \\
& \le C(\bs) \Big(\int_0^\infty \Big\|\sum_{k=1}^\infty e^{-2 \lambda_k r} \lambda_k^\theta e_k^2\Big\|_{\br/2} {\rm d}t\Big)^{\bs/2} \\
& \le  C(\bs) \Big(\sum_{k=1}^\infty \lambda_k^\theta \|e_k\|^2_\br 
\Big(\int_0^\infty e^{-2 \lambda_k r} {\rm d}t \Big) \Big)^{\bs/2} \\
& \le  C(\bs) \Big(\sum_{k=1}^\infty \frac1{2 \lambda_k^{1-\theta}} \Big)^{\bs/2}  <\infty.
\end{align*}  
This shows \eqref{wa1} with $\bs=\br=q+p-2$. 
\end{ex}

In the derivation of the existence of an invariant measure in Theorem \ref{tm-erg} when $q>2$, we need the following Sobolev regularity of the Ornstein--Uhlenbeck process $(W_A(t))_{t \ge 0}$.

\begin{ap} \label{ap-wa+}
There exists a constant $\beta_0 \in (0, 1)$ such that  
\begin{align} \label{wa+}
\sup_{t \ge 0} \ee \|W_A(t)\|_{\beta_0, p}<\infty. 
\end{align}  
\end{ap}

\begin{ex} \label{ex-wa+}
As in Example \ref{ex-wa},  
\begin{align*} 
\sup_{t \ge 0}\ee \|W_A(t)\|_{\beta_0, \br}^{\bs}
&\le C(\bs) \Big(\int_0^\infty \|(-A)^{\frac{\beta_0}2} S_r G\|^2_{\gamma(L^2, L^\br)} {\rm d}r\Big)^{\bs/2} \\
& \le  C(\bs) \Big(\sum_{k=1}^\infty \frac1{\lambda^{1-(\beta_0+\theta)}_k} \Big)^{\bs/2},
\end{align*} 
which is convergent if and only if $\beta_0+\theta<1/2$.
As $\theta<1/2$ in Example \ref{ex-wa}, this shows \eqref{wa+} with $\br=p$ and $\bs=1$ for any $\beta_0 \in (0, 1/2-\theta)$.
\end{ex}

We need the following standard elliptic condition to derive Wang-type Harnack inequalities and the ergodicity for $(P_t)$ in Section \ref{sec4}.

\begin{ap}\label{ap-ell}
$G G^*$ is invertible on $L^p$, with inverse $(G G^*)^{-1}$, such that $G^{-1}:=G^* (G G^*)^{-1}: L^p \to U$ is a bounded linear operator:
\begin{align} \label{G-1}
\|G^{-1}\|_\infty:=\sup_{x \in U:~ x \neq 0} \frac{\| G^{-1} x\|}{\|x\|_p}<\infty.
\end{align}
\end{ap}

\begin{rk}  \label{rk-G-1}
Let $\theta \in [0, 1/2)$.
Then the operator $G: =(-A)^{\theta/2}$ defined in Example \ref{ex-wa} is invertible with bounded inverse $G^{-1}: =(-A)^{-\theta/2}$. 
\end{rk}

\subsection{Main results}

Now, we are in a position to present our main results.
Let us first recall some definitions.

Let $x \in L^p$ and denote by $(X_t^x)_{t \ge 0}$ the mild solution of Eq. \eqref{rd} with the initial datum $X_0=x$ (see Lemma \ref{lm-lq} for the well-posedness). 
Then $(X^x_t)_{t \ge 0}$ is a Markov process that generates a Markov semigroup $(P_t)$ defined as 
\begin{align*}
P_t \phi(x):=\ee \phi(X^x_t), \quad 
t \ge 0, \ x \in L^p, \ \phi \in \BB_b(L^p).
\end{align*} 
The Markov semigroup $(P_t)$ is called of strong Feller if $P_t \phi \in \CC_b(L^p)$ for any $t>0$ and $\phi \in \BB_b(L^p)$.
A probability measure $\mu$ on $L^p$ is said to be an invariant measure of
of $(P_t)$ (or of Eq. \eqref{see}), if 
\begin{align*}
\int_{L^p} P_t \phi(x) \mu({\rm d}x)=\mu(\phi):=\int_{L^p} \phi(x) \mu({\rm d}x),
\quad \phi \in \BB_b(L^p), \ t \ge 0.
\end{align*} 
An invariant measure $\mu$ of $(P_t)$ is called \emph{ergodic}, if 
\begin{align} \label{df-erg}
\lim_{T \to \infty} \frac1T \int_0^T P_t \phi(x) {\rm d}t=\mu(\phi) 
\quad \text{in}~ L^2(L^p; \mu),\quad \forall~ \phi \in L^2(L^p; \mu).
\end{align}
It is well-known that if $(P_t)$ admits a unique invariant measure, it is  (uniquely) ergodic.

Our first main result is the following log-Harnack inequality and power-Harnack inequality, from which the uniqueness of the invariant measure, if it exists, for $(P_t)$ follows. 
Here we use $\chi$ to denote the indicator function.

\begin{tm} \label{main1}
Let Assumptions \ref{ap-f}, \ref{ap-wa}, and \ref{ap-ell} hold.
Define

\begin{align} \label{lam}
\lambda:= -L_f + \theta \chi_{q=2, p \neq 2}
+ \lambda_1 \chi_{q \neq 2, p=2}
+(\lambda_1+\theta) \chi_{q=p=2}.
\end{align} 
For any $T > 0$, $\bs >1$, $x, y \in L^p$, and $\phi \in \BB^+_b(L^p)$,
\begin{align}
& P_T \log \phi(y) \le \log P_T \phi(x)
+\frac{\lambda \|G^{-1}\|_\infty^2 \|x-y\|_p^2}{e^{2 \lambda T}-1}, \label{har-log} \\
& (P_T \phi(y))^\bs \le P_T \phi^\bs(x)
\exp\Big(\frac{\bs \lambda \|G^{-1}\|_\infty^2 \|x-y\|_p^2}{(\bs-1) (e^{2 \lambda T}-1)}\Big).  \label{har-pow}
\end{align}
Consequently, $(P_t)$ admits at most one invariant measure.
\end{tm}

Another main result is the existence of an invariant measure for $(P_t)$.
Combined with the uniqueness result in Theorem \ref{main1}, $(P_t)$ possesses exactly one ergodic invariant measure.

\begin{tm} \label{main2}
Let Assumptions \ref{ap-f}, \ref{ap-wa}, and \ref{ap-wa+} hold. 
Assume that $q>2$ such that 
\begin{align} \label{d}
d<\frac{2p(q+p-2)}{(p-1)(q-2)}.
\end{align}
Then $(P_t)$ has an invariant measure $\mu$ with full support on $L^p$ such that $\mu( \|\cdot\|_{q+p-2}^{q+p-2})<\infty$.
Assume furthermore that Assumption \ref{ap-ell} holds, then $\mu$ is the unique ergodic invariant measure of $(P_t)$.  
\end{tm}

\section{Coupling and Moments' Estimations}
\label{sec3}

The main aims of this section are to show the existence of a unique global solution to Eq. \eqref{rd} and to construct a well-defined coupling process for this solution process.
We also derive several uniform a priori estimates on moments of these two processes, which will be used in Section \ref{sec4} to derive Wang-type Harnack inequalities and the ergodicity of $(P_t)$.

\subsection{Well-posedness and moments' estimations}

Let us first recall that an $L^p$-valued process $(X_t)$ is a mild solution of Eq. \eqref{see} with the initial datum $X_0=x$ if $\pp$-a.s.
\begin{align}\label{mild} 
X_t &=S_t x+ \int_0^t S_{t-r} F(X_r) {\rm d}r+W_A(t),
\quad t \ge 0.
\end{align}
From Remark \ref{rk-F}, the deterministic convolution in Eq. \eqref{mild} makes sense.
Define $Z=X-W_A$.
It is clear that $X$ is a mild solution of Eq. \eqref{see} if and only if $Z$ is a mild solution of the random PDE 
\begin{align} \label{Z} 
\partial_t Z_t =\Delta Z_t+ F(Z_t+W_A(t)),
\quad Z_0=x.
\end{align}

The following results show the existence of a unique mild solution of Eq. \eqref{see}, which is a Markov process and depends on the initial data continuously in a pathwise sense, where the constant $\lambda$ is defined in \eqref{lam}.

\begin{lm} \label{lm-lq}
Let $T>0$, $x \in L^p$, and Assumptions \ref{ap-f}-\ref{ap-wa} hold.
Eq. \eqref{rd} with an initial datum in $L^p$ possesses a unique mild solution, in $\CC([0,T]; L^p) \cap L^{q+p-2}(0, T; L^{q+p-2})$ $\pp$-a.s., which is a Markov process.
Moreover, for all $t \ge 0$ and $x, y \in L^p$,        
\begin{align} \label{con}
& \|X_t^x-X_t^y\|_p \le e^{- \lambda t}  \|x-y\|_p, ~\pp\text{-a.s.} 
\end{align}  
\end{lm}

\begin{proof}
From \eqref{f+}, $f$ is locally Lipschitz continuous, so it is not difficult to show that both Eq. \eqref{Z} with $Z=X-W_A$ and Eq. \eqref{see} exist local solutions on $[0, T_0)$ for some stopping time $T_0 \in (0, T]$.
Extending this local solution to the whole time interval $[0, T]$ requires uniform a priori estimates for $Z$ and $X$.

As Eq. \eqref{Z} is a pathwise random PDE, we test $p |Z_t|^{p-2} Z_t$ on this equation with $t \in [0, T_0)$ and use the conditions \eqref{f}-\eqref{f+} and Young inequality to obtain
\begin{align*}  
& \partial_t \|Z_t\|_p^p 
+ p(p-1) \int_{\OOO} |Z_t|^{p-2} |\nabla Z_t|^2 {\rm d}\xi  \\
& =  p\< |Z_t|^{p-2}Z_t, F(Z_t+W_A(t))-F(W_A(t))\> 
+p \<|Z_t|^{p-2}Z_t, F(W_A(t)) \> \\
& \le  p L_f \|Z_t\|_p^p   
- p \theta \|Z_t\|_{q+p-2}^{q+p-2} 
+p \<|Z_t|^{p-2}Z_t, F(W_A(t)) \> \nonumber \\
& \le  p L_f \|Z_t\|_p^p  
- p \theta_1 \|Z_t\|_{q+p-2}^{q+p-2} 
+C \|F(W_A(t))\|_\frac{q+p-2}{q-1}^\frac{q+p-2}{q-1}\nonumber \\
& \le  C (1+\|W_A(t)\|_{q+p-2}^{q+p-2})
+p L_f \|Z_t\|_p^p  - p \theta_1 \|Z_t\|_{q+p-2}^{q+p-2},
\end{align*}  
where $\theta_1$ could be chosen as any positive number smaller than $\theta$.
This yields
\begin{align} \label{Zp} 
& \|Z_t\|_p^p 
+ p \theta_1 \int_0^t \|Z_r\|_{q+p-2}^{q+p-2} {\rm d} r 
+ p(p-1) \int_0^t \int_{\OOO} |Z_t|^{p-2} |\nabla Z_t|^2 {\rm d}\xi  {\rm d} r  \\
& \le  \|x\|_p^p+C \int_0^t (1+\|W_A(r)\|_{q+p-2}^{q+p-2}) {\rm d} r 
+p L_f \int_0^t \|Z_r\|_p^p  {\rm d} r. \nonumber
\end{align}
Using Gr\"onwall lemma, we obtain
\begin{align*} 
& \|Z_t\|_p^p 
+ p \theta_1 \int_0^t \|Z_r\|_{q+p-2}^{q+p-2} {\rm d} r 
+ p(p-1) \int_0^t \int_{\OOO} |Z_t|^{p-2} |\nabla Z_t|^2 {\rm d}\xi  {\rm d} r  \\ 
& \le e^{p L_f t} \Big(\|x\|_p^p+C \int_0^t (1+\|W_A(r)\|_{q+p-2}^{q+p-2}) {\rm d} r \Big).
\end{align*}

The above uniform estimate, in combination with the condition \eqref{wa1} of $W_A$ in Assumption \ref{ap-wa}, implies the global existence of a mild solution $Z$ to Eq. \eqref{Z} on $[0, T]$ in $\CC([0,T]; L^p) \cap L^{q+p-2}(0, T; L^{q+p-2})$ $\pp$-a.s.
Considering the relation $X=Z+W_A$ and the condition \eqref{wa1}, we obtain a global mild solution $X$ to Eq. \eqref{see}.

To show the continuous dependence \eqref{con}, let us note that  
\begin{align*} 
\partial_t (X_t^x-X_t^y) &=A (X_t^x-Y_t^y)+F(X_t^x)-F(Y_t^y).
\end{align*}
Testing $p |X_t^x-X_t^y|^{p-2} (X_t^x-X_t^y)$ on the above equation, using integration by parts formula, and applying the condition \eqref{f}, we obtain
\begin{align} \label{con-xy} 
& \|X_t^x-X_t^y\|_p^p 
+p \theta \int_0^t \|X_r^x-X_r^y\|_{q+p-2}^{q+p-2} {\rm d}r \nonumber \\
&\quad + p(p-1) \int_0^t \int_{\OOO} |X_r^x-X_r^y|^{p-2} |\nabla (X_r^x-X_r^y)|^2 {\rm d} \xi {\rm d}r  \nonumber \\
& \le \|x-y\|_p^p + p L_f \int_0^t \|X_r^x-X_r^y\|_p^p {\rm d}r.
\end{align}
We conclude \eqref{con} with $\lambda= -L_f$ by Gr\"onwall lemma.
When $q=2$, then \eqref{con} holds with $\lambda=\theta-L_f$, as one can subtract the first integral on the left-hand side of the above inequality; while $p=2$, then using the Poincar\'e inequality \eqref{poin} yields \eqref{con} with $\lambda=\lambda_1-L_f$.
Similarly, \eqref{con} holds with $\lambda=\lambda_1+\theta-L_f$ when $q=p=2$.
These statements show \eqref{con} with $\lambda$ given by \eqref{lam}.

The pathwise continuous dependence implies the uniqueness of the solution to Eq. \eqref{see}.
One can also show the Markov property for this solution using a standard method; see, e.g., \cite[Theorem 9.21]{PZ14}.
This completes the proof.
\end{proof}

\begin{rk}
The pathwise estimate \eqref{con} immediately yields an estimate between any two solutions in $\br$-Wasserstein distance for any $\br \ge 1$: let $(X^{x_i}_t)_{i=1}^2$ be the solutions to Eq. \eqref{see} starting from $(x_i)_{i=1}^2$ with laws $(\mu^i_0)_{i=1}^2$ on $L^p$, respectively, then 
\begin{align*}
{\rm W}_{\br}(\mu^1_t, \mu^2_t)
: =\inf (\ee \|X_t^{x_1}-X_t^{x_2}\|_p^{\br})^\frac1\br
\le e^{- \lambda t} {\rm W}_{\br}(\mu^1_0, \mu^2_0),
\quad t \ge 0,   
\end{align*} 
where the infimum runs over all random variables $(X^{x_i}_t)_{i=1}^2$ with laws $(\mu^i_t)_{i=1}^2$, $t \ge 0$.
A similar contraction-type estimate in $2$-Wasserstein distance on $\rr^d$ had been investigated in \cite{BGG12(JFA)}.
\end{rk}

\subsection{Construction of coupling and moments' estimations}

Let $T>0$ be fixed throughout the rest of Section \ref{sec3} and set
\begin{align} \label{gt} 
\gamma_t:=\int_0^{T-t} e^{2 \lambda r} {\rm d}r
=\frac{e^{2 \lambda (T-t)}-1}{2 \lambda},
\quad t \in [0, T],
\end{align} 
where $\lambda$ is given in \eqref{lam}.
For convention, if $\lambda=0$, we set $\frac{e^{2 \lambda t}-1}{2 \lambda}:=t$ for $t \in [0, T]$.
Then $\gamma$ is smooth, strictly positive, and strictly decreasing on $[0,T)$ (with $\gamma_T=0$) such that
\begin{align} \label{gt=}   
\gamma_t'+2 \lambda \gamma_t+1=0.
\end{align}
Moreover, the integrals of $\gamma^{-1}$ and $\gamma^{-2}$ on $[0, T)$ diverge:
\begin{align} \label{gt>}   
\int_0^T \frac1{\gamma_t} {\rm d}t=\infty, \quad 
\int_0^T \frac1{\gamma_t^2} {\rm d}t=\infty.
\end{align}

Now, we can define the coupling $Y$ of $X$ as the mild solution of the coupling equation 
\begin{align} \label{rd-y}   
{\rm d}Y_t =( AY_t+F(Y_t)+\gamma_t^{-1} (X_t-Y_t) ){\rm d}t+G {\rm d}W_t, 
\end{align}
with an initial datum $Y_0=y \in L^p$. 
Since the additional drift term $\gamma_t^{-1} (X_t-Y_t)$ is Lipschitz continuous for each fixed $t \in [0, T)$ and $\omega \in \Omega$, one can use similar arguments in Lemma \ref{lm-lq} to show that $Y$ is a well-defined continuous process on $[0, T)$.

\begin{rk} \label{rk-Y}
As $\gamma^{-1}$ is continuous and thus integrable on $[0, T_0] \subset [0, T)$ for any $T_0 \in (0, T)$, one can use the arguments in Lemma \ref{lm-lq} to extend the local solution to $[0, T)$.
However, it is difficult to get a uniform a priori estimation, following the idea in Lemma \ref{lm-lq}, to conclude the well-posedness of $Y$ at $T$ as $\gamma_T^{-1}$ satisfies \eqref{gt>}. 
\end{rk}

For each $s \in [0, T)$, we set  
\begin{align} \label{v} 
v_s:=\frac{G^{-1} (X_s-Y_s)}{\gamma_s}, \quad 
\widetilde{W}_s:=W_s+\int_0^s v_r {\rm d}r,
\end{align}
and define
\begin{align} \label{M}
M_s: & =\exp\Big( -\int_0^s (v_r, {\rm d}W_r)
-\frac12 \int_0^s \|v_r\|^2 {\rm d}r\Big).
\end{align} 
From \eqref{v} and \eqref{M}, $M$ can also be rewritten as  
\begin{align} \label{M+}
M_s: & =\exp\Big( -\int_0^s (v_r, {\rm d}{\widetilde W}_r)
+\frac12 \int_0^s \|v_r\|^2 {\rm d}r\Big), \quad s \in [0, T).
\end{align} 

By the representation \eqref{v} and the nondegenerate condition \eqref{G-1}, we have 
\begin{align} \label{est-v} 
\frac12 \int_0^s \|v_r\|^2 {\rm d}r
\le \frac{ \|G^{-1}\|_\infty^2}2 \int_0^s \frac{\|X_r-Y_r\|_p^2}{\gamma_r^2} {\rm d}r.
\end{align}
We first show that for any $s \in (0, T)$, 
$${\mathbb Q}_s:=M_s \pp$$ 
is a probability measure, of which $\ee_s$ denotes the expectation, ensured by a Novikov condition.
Moreover, $(\widetilde W_t)_{t \in [0, s]}$ is a $U$-valued cylindrical Wiener process under the probability measure ${\mathbb Q}_s$ through Girsanov theorem.

\begin{lm} \label{lm-R-}
Let Assumptions \ref{ap-f}, \ref{ap-wa}, and \ref{ap-ell} hold.
For any $s \in (0, T)$, ${\mathbb Q}_s$ is a probability measure and $(\widetilde W_t)_{t \in [0, s]}$ is a $U$-valued cylindrical Wiener process under ${\mathbb Q}_s$.
\end{lm}

\begin{proof} 
It follows from Eq. \eqref{see} and Eq. \eqref{rd-y} on $[0, s]$ that   
\begin{align} \label{x-y-p}
\partial_t (X_t-Y_t) &=A (X_t-Y_t)+F(X_t)-F(Y_t)
- \gamma_t^{-1} (X_t-Y_t), \quad \pp\text{-a.s.} 
\end{align} 
As in the proof of the inequality \eqref{con-xy}, we test $p |X_t-Y_t|^{p-2} (X_t-Y_t)$ on the above equation, use the integration by parts formula, and apply the condition \eqref{f} to obtain 
\begin{align*} 
& \partial_t \|X_t-Y_t\|_p^p 
+ p(p-1) \int_{\OOO} |X_t-Y_t|^{p-2} |\nabla (X_t-Y_t)|^2 {\rm d} \xi
\\
& \le p L_f \|X_t-Y_t\|_p^p - p \theta \|X_t-Y_t\|_{q+p-2}^{q+p-2}
- p\gamma_t^{-1} \|X_t-Y_t\|_p^p.
\end{align*} 
It follows from the chain rule that 
\begin{align*} 
& \partial_t \|X_t-Y_t\|_p^2
=\frac2p \|X_t-Y_t\|_p^{2-p}\partial_t \|X_t-Y_t\|_p^p \\
& \le - 2(p-1) \|X_t-Y_t\|_p^{2-p} \int_{\OOO} |X_t-Y_t|^{p-2} |\nabla (X_t-Y_t)|^2 {\rm d} \xi \\
& +2 L_f \|X_t-Y_t\|^2_p 
- 2 \theta \|X_t-Y_t\|_p^{2-p} \|X_t-Y_t\|_{q+p-2}^{q+p-2}
- 2 \gamma_t^{-1}\|X_t-Y_t\|^2_p,
\end{align*} 
and thus
\begin{align*} 
\partial_t \|X_t-Y_t\|_p^2 
\le -2 \lambda \|X_t-Y_t\|^2_p  - 2 \gamma_t^{-1}\|X_t-Y_t\|^2_p.
\end{align*} 
The product rule of differentiation and the equality \eqref{gt=} yield that 
\begin{align*} 
\partial_t (\gamma_t^{-1} \|X_t-Y_t\|_p^2 )
& =\gamma_t^{-1} \partial_t \|X_t-Y_t\|_p^2
-\gamma_t^{-2} \gamma'_t \|X_t-Y_t\|_p^2 \\
& \le -\gamma_t^{-2} (\gamma'_t+2 \lambda \gamma_t+2) \|X_t-Y_t\|_p^2 \\
& = - \gamma_t^{-2} \|X_t-Y_t\|_p^2.
\end{align*} 
Integrating on both sides from $0$ to $s$, we obtain 
\begin{align} \label{x-y} 
\frac{\|X_s-Y_s\|_p^2}{\gamma_s}
+\int_0^s \frac{\|X_t-Y_t\|_p^2}{\gamma_t^2} {\rm d}t 
\le \frac{\|x-y\|_p^2}{\gamma_0}, \quad \pp\text{-a.s.}
\end{align} 
This pathwise estimate, in combination with the estimate \eqref{est-v}, particularly implies the Novikov condition
\begin{align} \label{nov} 
\ee \exp \Big(\frac12 \int_0^s \|v_t\|^2 {\rm d}t \Big)
\le \exp \Big(\frac{\|G^{-1}\|_\infty^2 \|x-y\|_p^2}{2 \gamma_0}\Big)
<\infty.
\end{align} 
This shows that $\ee M_s=1$ and thus ${\mathbb Q}_s:=M_s \pp$ is a probability measure on $\FFF_s$ equivalent to $\pp$.
By Girsanov theorem, $(\widetilde W_t)_{t \in [0, s]}$ is a $U$-valued cylindrical Wiener process under ${\mathbb Q}_s$. 
\end{proof}

We note that the Novikov condition \eqref{nov} also holds with $s=T$, according to the estimates \eqref{est-v} and \eqref{x-y} with $s=T$.
Therefore, ${\mathbb Q}:=M_T \pp$ is a probability measure on $\FFF_T$ equivalent to $\pp$. We denote by $\ee_{\mathbb Q}$ the expectation concerning ${\mathbb Q}$.

Next, we will give two uniform moments' estimations for certain functionals of $(M_s)$.

\begin{lm} \label{lm-R}
Let Assumptions \ref{ap-f}, \ref{ap-wa}, and \ref{ap-ell} hold.
Then 
\begin{align} \label{est-R}
& \sup_{s \in [0, T)} \ee[M_s \log M_s]  
\le \frac{\lambda \|G^{-1}\|_\infty^2}{e^{2 \lambda T}-1}\|x-y\|_p^2,
\quad x, y \in L^p. 
\end{align} 
Consequently, $M_T:=\lim_{s \uparrow T} M_s$ exists and $(M_s)_{s \in [0, T]}$ is a well-defined uniformly integrable martingale (under $\pp$).
\end{lm}

\begin{proof} 
Let $s \in [0, T)$ be fixed.
By the construction \eqref{v}, we can rewrite Eq. \eqref{see} and Eq. \eqref{rd-y} on $[0, s]$ as  
\begin{align} 
{\rm d}X_t &=( AX_t+F(X_t)-\gamma_t^{-1}(X_t-Y_t) ) {\rm d}t +G {\rm d}{\widetilde W}_t, \label{rd-x+}   \\
{\rm d}Y_t &=(AY_t+F(Y_t)) {\rm d}t+G {\rm d}{\widetilde W}_t, \label{rd-y+}  
\end{align}
with initial values $X_0=x$ and $Y_0=y$, respectively.
Then Eq. \eqref{x-y-p} about $X-Y$ also holds ${\mathbb Q}_s$-a.s. by the equivalence between $\pp$ and ${\mathbb Q}_s$ on $\FFF_s$.  
Therefore, the pathwise estimate \eqref{x-y} is valid ${\mathbb Q}_s$-a.s., which in combination with the equality \eqref{M+} and the estimate \eqref{est-v} implies that
\begin{align} \label{log-m}
\log M_s & =-\int_0^s (v_r, {\rm d}{\widetilde W}_r)
+\frac12 \int_0^s \|v_r\|^2 {\rm d}r  \nonumber \\
& \le -\int_0^s (v_r, {\rm d}{\widetilde W}_r)
+\frac{\|G^{-1}\|_\infty^2 \|x-y\|_p^2}{2 \gamma_0}, \quad {\mathbb Q}_s\text{-a.s.} 
\end{align}  
Taking into account the fact in Lemma \ref{lm-R-} that $(\widetilde W_t)_{t \in [0, s]}$ is a $U$-valued cylindrical Wiener process under ${\mathbb Q}_s$, we arrive at
\begin{align*}  
& \ee[M_s \log M_s]  
=\ee_s \log M_s  
\le \frac{\|G^{-1}\|_\infty^2 \|x-y\|_p^2}{2 \gamma_0},   
\end{align*} 
and thus obtain \eqref{est-R}, noting that $\gamma_0$ is given in \eqref{gt} with $t=0$.
By the martingale convergence theorem, $M_T:=\lim_{s \uparrow T} M_s$ exists and $(M_t)_{t \in [0, T]}$ is a martingale.
\end{proof}

Lemmas \ref{lm-R-}-\ref{lm-R} ensures that $(\widetilde W_t)_{t \in [0, T]}$ is a $U$-valued cylindrical Wiener process under the probability measure ${\mathbb Q}$ and
\begin{align} \label{est-R-T}
& \sup_{s \in [0, T]} \ee[M_s \log M_s]  
\le \frac{\lambda \|G^{-1}\|_\infty^2}{e^{2 \lambda T}-1}\|x-y\|_p^2,
\quad x, y \in L^p. 
\end{align}

\begin{lm} \label{lm-R+}
Let Assumptions \ref{ap-f}, \ref{ap-wa}, and \ref{ap-ell} hold.
For any $x, y \in L^p$ and $\bs>1$, 
\begin{align} \label{est-R+}  
& \sup_{s \in [0, T]} \ee M_s^{\frac \bs{\bs-1}} 
\le \exp\Big(\frac{\bs \lambda \|G^{-1}\|_\infty^2 \|x-y\|_p^2}{(\bs-1)^2 (e^{2 \lambda T}-1)}\Big).
\end{align}  
\end{lm}

\begin{proof}
Let $s \in [0, T]$.
Denote by $v^\bs_r:=-\frac1{\bs-1} v_r$ for $r \in [0, s] \subset [0, T]$.
The representation \eqref{M+} and the pathwise estimate \eqref{x-y} with $\gamma_0$ given in \eqref{gt} yield that  
\begin{align*}
M_s^{\frac1{\bs-1}}
& =\exp \Big( -\frac1{\bs-1}\int_0^s \<v_r, {\rm d} {\widetilde W}_r\> + \frac1{2(\bs-1)} \int_0^s \|v_r\|^2 {\rm d}r \Big)  \\
& =\exp \Big( \int_0^s \<v^\bs_r, {\rm d}{\widetilde W}_r \>
- \frac12 \int_0^s \|v^\bs_r\|^2 {\rm d}r \Big) \\
& \qquad \times \exp \Big(\frac\bs{2(\bs-1)^2} \int_0^s \| v_r \|^2 {\rm d}r \Big) \\ 
& \le {\widetilde M}_s \exp\Big(\frac{\bs \lambda \|G^{-1}\|_\infty^2 \|x-y\|_p^2}{(\bs-1)^2 (e^{2 \lambda T}-1)}\Big), \quad {\mathbb Q}_s\text{-a.s.} 
\end{align*}  
where ${\widetilde M}_s:=\exp( \int_0^s \<v^\bs_r, {\rm d}{\widetilde W}_r \>
- \frac12 \int_0^s \|v^\bs_r\|^2 {\rm d}r)$.
It follows that   
\begin{align*}
\ee M_s^{\frac \bs{\bs-1}}
=\ee_s M_s^{\frac1{\bs-1}} 
\le \exp\Big(\frac{\bs \lambda \|G^{-1}\|_\infty^2 \|x-y\|_p^2}{(\bs-1)^2 (e^{2 \lambda T}-1)}\Big) \ee_s {\widetilde M}_s.
\end{align*} 
Taking into account the fact that $(\widetilde W_t)_{t \in [0, T]}$ is a $U$-valued cylindrical Wiener process under ${\mathbb Q}$ and the Novikov condition that 
\begin{align*} 
\ee_s \exp \Big(\frac12 \int_0^s \|v^\bs_r\|^2 {\rm d}r\Big)
& =\ee_s \exp \Big(\frac1{2(\bs-1)^2} \int_0^s \|v_r\|^2 {\rm d}r\Big) \\
& \le \exp \Big(\frac{\|G^{-1}\|_\infty^2 \|x-y\|_p^2}{2(\bs-1)^2 \gamma_0}\Big) 
<\infty, 
\end{align*} 
we have $\ee_s {\widetilde M}_s=1$ and obtain \eqref{est-R+}.   
\end{proof}

\section{Harnack Inequalities and Ergodicity}
\label{sec4}

In the last section, we derive Harnack inequalities and the ergodicity for the Markov semigroup $(P_t)$. 
In the first two parts, we give the proof of our main results, Theorems \ref{main1} and \ref{main2}, respectively.
Other applications, inclusive of several estimates for the density of $(P_t)$, are also derived.

\subsection{Harnack inequalities}

We begin with the following Harnack inequalities.

\begin{tm} \label{tm-har}
Let Assumptions \ref{ap-f}, \ref{ap-wa}, and \ref{ap-ell} hold.
Then \eqref{har-log}-\eqref{har-pow} hold for any $T > 0$, $\bs >1$, $x, y \in L^p$, and $\phi \in \BB^+_b(L^p)$.
\end{tm}

\begin{proof}
We first show that $X_T=Y_T$, ${\mathbb Q}$-a.s.
From Lemma \ref{lm-R}, $({\widetilde W}_t)_{[0, T]}$ is a $U$-valued cylindrical Wiener process under ${\mathbb Q}$.
So $Y_t$ can be solved up to time $T$.
Let 
\begin{align*}
\tau:=\inf\{t \in [0, T]: \ X_t=Y_t\}
\quad \text{with} \quad \inf \emptyset : =\infty.
\end{align*}

Suppose that for some $\omega \in \Omega$ such that $\tau(\omega)>T$, then the continuity of the process $X-Y$, in Lemma \ref{lm-lq} and Remark \ref{rk-Y}, yields 
\begin{align*} 
\inf_{t \in [0, T]} \|X_t-Y_t\|_p^2(\omega)>0.
\end{align*}
By the divergence relation \eqref{gt>}, 
\begin{align*}  
\int_0^T \frac{\|X_t-Y_t\|_p^2(\omega)}{\gamma_t^2} {\rm d}t=\infty 
\end{align*}
holds on the set $(\tau>T):=\{\omega: \tau(\omega)>T\}$.
But according to the pathwise estimate \eqref{x-y} which holds ${\mathbb Q}$-a.s. and the fact that $\pp$ and ${\mathbb Q}$ are equivalent on $\FFF_T$,
\begin{align*}  
\ee_{\mathbb Q} \int_0^T \frac{\|X_t-Y_t\|_p^2(\omega)}{\gamma_t^2} {\rm d}t
\le \frac{\|x-y\|_p^2}{\gamma_0}<\infty.
\end{align*}  
It follows from the above two estimates that ${\mathbb Q}(\tau>T)=0$, i.e., $\tau \le T$ ${\mathbb Q}$-a.s.
This shows that, for almost all $\omega$, there exists $t \in [0, T]$ such that $X_t(\omega) = Y_t(\omega)$. By the pathwise uniqueness of Eq. \eqref{see} and Eq. \eqref{rd-y}, $X_T=Y_T$ ${\mathbb Q}$-a.s.  
Therefore, we get a coupling $(X,Y)$ by the change of measure, with changed probability ${\mathbb Q}= M_T \pp$, such that $X_T=Y_T$, ${\mathbb Q}$-a.s.
Consequently, the inequalities \eqref{har-log}-\eqref{har-pow} follow from established results (see, e.g., \cite[Theorem 1.1.1]{Wan13}):
\begin{align*}
P_T \log \phi(y) & \le \log P_T \phi(x)+\ee [M_T \log M_T], \\
(P_T \phi(y))^\bs & \le P_T \phi^\bs(x) (\ee M_T^{\frac \bs{\bs-1}})^{\bs-1},
\end{align*}
and the estimations \eqref{est-R} and \eqref{est-R+} in Lemmas \ref{lm-R} and \ref{lm-R+}, respectively.
\end{proof}

The log-Harnack inequality \eqref{har-log} shows the strong Feller property of $(P_t)$ and implies the following gradient estimate and regularity properties for the Markov semigroup $(P_t)$. {\color{red} Here we denote $|\nabla \phi|(x)=\limsup_{y \rightarrow x} |\phi(y)-\phi(x)|/\|y-x\|$ for $\phi \in \BB_b(L^p)$ and $x \in L^p$}.

\begin{cor} \label{cor-fel}
Let Assumptions \ref{ap-f}, \ref{ap-wa}, and \ref{ap-ell} hold.
For any $T>0$, $x \in L^p$, and $\phi \in \BB_b(L^p)$,   
\begin{align}\label{est-gra}
{\color{red} |\nabla P_T \phi|(x)} \le \sqrt{\frac{2 \lambda \|G^{-1}\|_\infty^2}{e^{2 \lambda T}-1}} \sqrt{P_T \phi^2(x)-(P_T \phi(x))^2}.
\end{align} 
Consequently, $(P_t)$ admits at most one invariant measure, and if it has one, the density of $(P_t)$ relative to the invariant measure is strictly positive.
\end{cor}

\begin{proof}
The gradient estimate \eqref{est-gra} and the uniqueness of the invariant measure for $(P_t)$ with a strictly positive density, if it exists, are direct consequences of the log-Harnack inequality \eqref{har-log}, see Proposition 1.3.8 and Theorem 1.4.1 in \cite{Wan13}, respectively.  
\end{proof}

\begin{rk} 
The uniqueness of the invariant measure, if it exists, for 1D stochastic heat equation (Eq. \eqref{see} with Lipschitz coefficients) driven by white noise on $L^p(0, 1)$ with $p>4$ was shown in \cite{BR16(DCDS)}. 
So Corollary \ref{cor-fel} can be seen as filling the gap for $p \in (2, 4]$ in the additive white noise case.
\end{rk}

\begin{rk} 
Under the conditions in Corollary \ref{cor-fel}, there exist positive constants $C, T_0$ such that
\begin{align} \label{tv}
\|\LL(X_t^x)-\LL(X_t^y)\|_{\rm TV} \le C e^{-\lambda t} \|x-y\|_p,
\quad t \ge T_0,  \ x, y \in L^p,
\end{align}
where $\LL(X)$ denotes the distribution of $X$ on $L^p$, $\lambda$ is given in \eqref{lam}, and $\|\cdot \|_{\rm TV}$ denotes the total variation norm betweem two signed measures:
\begin{align*}
\|\mu-\nu\|_{\rm TV}:=\sup_{\|\phi\|_\infty \le 1}|\int_{L^p} \phi {\rm d}\mu-\int_{L^p} \phi {\rm d}\nu|,
\end{align*} 
for two signed measures $\mu$ and $\nu$, with  
$\|\phi\|_\infty:=\sup_{x \in L^p} |\phi(x)|$.
\end{rk}

\begin{proof}[Proof of Theorem \ref{main1}]
Theorem \ref{main1} follows from Theorem \ref{tm-har} and Corollary \ref{cor-fel}.
\end{proof}

\subsection{Ergodicity}

In this part, we show the existence of an invariant measure for the Markov semigroup $(P_t)$.
In combination with the uniqueness of the invariant measure, as shown in Corollary \ref{cor-fel}, we derive the existence of a unique and, thus, ergodic invariant measure. 
We also note that \cite[Theorem 6.1]{BG99(SPA)} used the factorization approach to obtain the existence of an invariant measure in $L^p$ with $p \ge 2$ under the martingale solution framework.

\begin{tm} \label{tm-erg}
Let Assumptions \ref{ap-f}, \ref{ap-wa}, and \ref{ap-wa+} hold. 
Assume that $q>2$ and \eqref{d} holds.
Then $(P_t)$ has an invariant measure $\mu$ with full support on $L^p$ such that $\mu( \|\cdot\|_{q+p-2}^{q+p-2})<\infty$.
Assume furthermore that Assumption \ref{ap-ell} holds, then $\mu$ is the unique and thus ergodic invariant measure of $(P_t)$. 
\end{tm}

\begin{proof}
The uniqueness of the invariant measure and the strong Feller Markov property for $(P_t)$ have been shown in Corollary \ref{cor-fel}.
Thus, to show the existence of an invariant measure, by Krylov--Bogoliubov theorem, it suffices to verify the tightness of the sequence of probability measures $(\mu_n)$ defined by
\begin{align} \label{mun}
\mu_n:=\frac1n \int_0^n \delta_0 P_t {\rm d}t, \quad n \in \nn_+,
\end{align}
where $\delta_0 P_t$ is the distribution of $X_t^0$, the solution of Eq. \eqref{see} with the initial datum $X_0=0$.

It follows from the relation $X=Z+W_A$, the estimate \eqref{Zp} with $x=0$, and Young inequality that
\begin{align} \label{est-tig} 
& \|X^0_t\|_p^p 
\le 2^{p-1} \|Z^0_t\|_p^p+2^{p-1}\|W_A(t)\|_p^p \nonumber  \\
& \le  C \int_0^t (1+\|W_A(r)\|_{q+p-2}^{q+p-2}) {\rm d} r
+2^{p-1} p L_f \int_0^t \|Z^0_r\|_p^p  {\rm d} r  \nonumber \\
& \quad - 2^{p-1} p \theta_1 \int_0^t \|Z^0_r\|_{q+p-2}^{q+p-2} {\rm d} r +2^{p-1} \|W_A(t)\|_p^p \nonumber \\
& \le  C \int_0^t (1+\|W_A(r)\|_{q+p-2}^{q+p-2}) {\rm d} r
- \theta_4 \int_0^t \|Z^0_r\|_{q+p-2}^{q+p-2} {\rm d} r
 +2^{p-1}\|W_A(t)\|_p^p \nonumber \\
 & \le  C \int_0^t (1+\|W_A(r)\|_{q+p-2}^{q+p-2}) {\rm d} r
- \theta_5 \int_0^t \|X^0_r\|_{q+p-2}^{q+p-2} {\rm d} r
 +2^{p-1}\|W_A(t)\|_p^p,   
\end{align}
for some constants $\theta_4, \theta_5>0$, where we have used the elementary inequality $|\xi-\eta|^\br \ge 2^{1-r} \xi^r-\eta^r$ for $\xi, \eta \ge 0$ and $\br \ge 1$, in the last inequality. 
Then we have
\begin{align*} 
& \theta_5 \int_0^t \|X^0_r\|_{q+p-2}^{q+p-2} {\rm d} r
\le C \int_0^t (1+\|W_A(r)\|_{q+p-2}^{q+p-2}) {\rm d} r
+ 2^{p-1} \|W_A(t)\|_p^p.
\end{align*}
The above estimate, in combination with the condition \eqref{wa1}, yields that there exists a constant $C$ such that for all $n \ge 1$, 
\begin{align} \label{est-tig1}
\mu_n(\|\cdot\|_{q+p-2}^{q+p-2})
& =\frac1n \int_0^n \ee \|X^0_r\|_{q+p-2}^{q+p-2} {\rm d} r 
\nonumber \\
& \le \frac C{\theta_5} 
\Big( 1+\frac{\ee \|W_A(n)\|_p^p}{n}+
\frac1n \int_0^n \ee \|W_A(r)\|_{q+p-2}^{q+p-2} {\rm d} r \Big) \nonumber \\
& \le C.
\end{align}

It follows from the ultracontractivity \eqref{ult}, with $\br=p$ and $\bs=\frac{q+p-2}{q-1}$, and Minkovski and Young inequalities that 
\begin{align*} 
& \int_0^n \Big\| \int_0^t S_{t-r} F(X^0_r) {\rm d}r \Big\|_{\beta, p} {\rm d}t   \\
 & \le C \int_0^n \int_0^t e^{-\lambda_1 (t-r)}
  (t-r)^{-\alpha}(1+\|X^0_r\|^{q-1}_{q+p-2}) {\rm d}r {\rm d}t  \\
   & \le C \Big(\int_0^n e^{-\lambda_1 t} t^{-\alpha}{\rm d}t \Big) 
\Big(\int_0^n  (1+\|X^0_t\|^{q-1}_{q+p-2}) {\rm d}t \Big),
\end{align*} 
where $\alpha=\frac\beta 2+\frac{d(p-1)(q-2)}{2p(q+p-2)} \in (0, 1)$ provided that $\beta>0$ is sufficiently small, since $d<\frac{2p(q+p-2)}{(p-1)(q-2)}$.
The fact that  
\begin{align*} 
& \sup_{n \ge 1}\Big(\int_0^n e^{-\lambda_1 t} t^{-\alpha}{\rm d}t \Big) 
= \int_0^\infty e^{-\lambda_1 t} t^{-\alpha}{\rm d}t
<\infty,  
\end{align*}
for all $\lambda_1>0$ and $\alpha \in (0, 1)$, and Young inequality imply that 
\begin{align*} 
& \int_0^n \Big\| \int_0^t S_{t-r} F(X^0_r) {\rm d}r \Big\|_{\beta, p} {\rm d}t
\le C \int_0^n  (1+\|X^0_t\|^{q+p-2}_{q+p-2}) {\rm d}t.
\end{align*}
By Fubini theorem, the estimate \eqref{est-tig1}, and the condition \eqref{wa+}, we arrive at 
\begin{align} \label{est-tig2}
\mu_n(\|\cdot\|_{\beta,p})
& =\frac1n \int_0^n \ee \|X^0_r\|_{\beta,p} {\rm d} r \nonumber \\
& \le \frac1n \ee \int_0^n \Big\| \int_0^t S_{t-r} F(X^0_r) {\rm d}r \Big\|_{\beta, p} {\rm d} r
+\frac1n \int_0^n \ee \|W_A(t)\|_{\beta,p} {\rm d} r \nonumber \\
& \le \frac Cn \int_0^n  (1+ \ee \|X^0_t\|^{q+p-2}_{q+p-2}) {\rm d}t 
+\frac1n \int_0^n \ee \|W_A(t)\|_{\beta,p} {\rm d} r \nonumber \\
& \le C,  
\end{align}
for all $n \ge 1$ and $\beta<(1-\frac{d(p-1)(q-2)}{2p(q+p-2)}) \wedge \beta_0$. 
For any fixed $p \ge 2$, we take $\beta<(1-\frac{d(p-1)(q-2)}{2p(q+p-2)}) \wedge \beta_0 \wedge \frac d p$, so that the embedding $W^{\beta,p} \subset L^p$ in \eqref{emb} is compact.  
Consequently, the above estimate \eqref{est-tig2} shows that 
$\{ u \in L^p: \ \|u\|_{\beta, p} \le N \}$ is relatively compact in $L^p$ for any $N>0$, and thus $(\mu_n)$ is tight.
This shows that an invariant measure, denoted by $\mu$, of $(P_t)$ exists.

To show that the invariant measure $\mu$ has full support on $L^p$, let us choose $\bs=2$, $\phi =\chi_\Gamma$, in \eqref{har-pow}, with $\Gamma$ being a Borel set in $L^p$, and get 
\begin{align*}
& (P_T \chi_\Gamma(x))^2 \int_{L^p} \exp\Big(-\frac{2 \lambda \|G^{-1}\|_\infty^2}{e^{2 \lambda T}-1} \|x-y\|_p^2\Big) \mu({\rm d} y) \\
& \le \int_{L^p} P_T \chi_\Gamma(y) \mu({\rm d} y)
=\int_{L^p} \chi_\Gamma(y) \mu({\rm d} y)=\mu(\Gamma), 
\quad T>0, \ x \in L^p.
\end{align*}
This shows that the transition kernel of $(P_t)$ is absolutely continuous with respect to $\mu$ so that it has a density $p_T(x, y)$. 
Suppose that ${\rm supp} \ \mu\neq L^p$, then there exist $x_0 \in L^p$ and $r>0$ such that $\mu(B(x_0, r))=0$, where $B(x_0, r)$ is a ball in $L^p$ with radius $r$ and center $x_0$.
Then $p_T(x_0, B(x_0, r))=0$ and 
$\pp(\|X_T^{x_0}-x_0\|_p \le r)=0$ for all $T>0$.
This contradicts the fact that $X_T^{x_0}$ is a continuous process on $L^p$ as shown in Lemma \ref{lm-lq}.

Similarly to \eqref{est-tig1}, we have (with $n=1$ and $X_0=x$)
\begin{align*}  
\int_0^1 P_t \|\cdot\|_{q+p-2}^{q+p-2}(x) {\rm d}t
=\int_0^1 \ee \|X^x_t\|_{q+p-2}^{q+p-2} {\rm d} t
\le C (1+\|x\|_p^p).
\end{align*}
Integrating on $L^p$ concerning the invariant measure $\mu$ and using the Fubini theorem, we obtain 
\begin{align*} 
\mu( \|\cdot\|_{q+p-2}^{q+p-2})  
=\int_0^1 \int_{L^p} P_t \|\cdot\|_{q+p-2}^{q+p-2}(x)\mu({\rm d}x) {\rm d}t 
\le C (1+\mu(\|\cdot\|_p^p))<\infty._{L^p}
\end{align*}
This shows that $\mu( \|\cdot\|_{q+p-2}^{q+p-2})<\infty$ and completes the proof.
\end{proof}

\begin{rk} 
In the case $q>2=p$, the condition \eqref{d} is equivalent to $d<4+8/(q-2)$, which will always be valid in $d=1,2,3$-dimensional cases. 
\end{rk}

\begin{rk} \label{rk-erg}
Under Assumptions \ref{ap-f}-\ref{ap-wa}, if $\lambda$ defined in \eqref{lam} is positive, one could expect to use the remote start method in, e.g., \cite[Theorem 6.3.2]{PZ96} to show that $(P_t)$ has an invariant measure $\mu$, once similar estimates as \eqref{con}-\eqref{Zp} are derived, without the restriction \eqref{d}. 
\end{rk}

\begin{rk}
Under the conditions of Theorem \ref{tm-erg} or Remark \ref{rk-erg}, $(P_t)$ has a unique invariant measure $\mu$ with full support on $L^p$, which shows that $(P_t)$ is irreducible, i.e., $P_T\chi_{\Gamma}(x)>0$ for arbitrary non-empty open set $\Gamma \subset L^p$, $x \in L^p$, and $T>0$.
Indeed, the power-Harnack inequality \eqref{har-pow} with $f=\chi_\Gamma$ yields that 
\begin{align} \label{irr}
 (P_T \chi_\Gamma(y))^\bs \le P_T \chi_\Gamma(x)
\exp\Big(\frac{\bs \lambda \|G^{-1}\|_\infty^2 \|x-y\|_p^2}{(\bs-1) (e^{2 \lambda T}-1)}\Big),
\quad y \in L^p.
\end{align}
The facts that $\mu$ is $P_T$-invariant and has full support on $L^p$ imply
\begin{align*}
\int_{L^p} P_T \chi_\Gamma(y) \mu({\rm d}y) 
\le \int_{L^p} \chi_\Gamma(y) \mu({\rm d}y) 
=\mu(\Gamma)>0,   
\end{align*}
which shows that there is a $y \in L^p$ such that $P_T \chi_\Gamma(y)>0$.
Then \eqref{irr} yields that $P_T \chi_\Gamma(x)>0$ for all $x \in L^p$ so that the irreducibility holds.
\end{rk}

\begin{proof}[Proof of Theorem \ref{main2}]
Theorem \ref{main2} follows from Theorem \ref{tm-erg} and Remark \ref{rk-erg}.
\end{proof}

\subsection{Estimates of density}

Finally, we use the Harnack inequalities \eqref{har-log}-\eqref{har-pow} to derive an estimate of the density, denoted by $p_T(x, \cdot)$, concerning the invariant measure $\mu$ of $(P_t)$.

\begin{cor}  \label{cor-den}
Let Assumptions \ref{ap-f}-\ref{ap-wa} hold.
Assume that $q>2$ such that \eqref{d} and Assumption \ref{ap-wa+} hold, or $\lambda$ defined in \eqref{lam} is positive.
Then for all $T>0$, $x \in L^p$, and $\bs >1$,  
\begin{align}\label{den}
& \|p_T (x, \cdot)\|_{L^\bs(\mu)} 
\le \Big( \int_{L^p} \exp\Big(-\frac{\bs \lambda \|G^{-1}\|_\infty^2 \|x-y\|_p^2}{e^{2 \lambda T}-1} \Big) \mu({\rm d}y)  \Big)^{-\frac{\bs-1}\bs}.
\end{align} 
\end{cor}

\begin{proof}
Let $T>0$ and $x \in L^p$.
For any $\bs>1$ and $\phi \in \mathcal G:=\{\psi \in B_b^+(L^p):~ \mu(\psi^\frac\bs{\bs-1}) \le 1\}$, it follows from \eqref{har-pow} with $\bs$ replaced by $\frac\bs{\bs-1}$ that 
\begin{align*}
(P_T \phi(x))^\frac\bs{\bs-1} 
\le P_T \phi^\frac\bs{\bs-1}(y)
\exp\Big(\frac{\bs \lambda \|G^{-1}\|_\infty^2 \|x-y\|_p^2}{e^{2 \lambda T}-1}\Big), \quad y \in L^p.  
\end{align*}
Noting that $\mu$ is $P_T$-invariant, we have
\begin{align*}
(P_T \phi(x))^\frac\bs{\bs-1}  \int_{L^p} \exp\Big(-\frac{\bs \lambda \|G^{-1}\|_\infty^2 \|x-y\|_p^2}{e^{2 \lambda T}-1}\Big) \mu({\rm d}y)
\le \mu(P_T \phi^\frac\bs{\bs-1})
=\mu(\phi^\frac\bs{\bs-1})=1,  
\end{align*}
from which we obtain
\begin{align*}
P_T \phi(x)  
\le \Big(\int_{L^p} \exp\Big(-\frac{\bs \lambda \|G^{-1}\|_\infty^2 \|x-y\|_p^2}{e^{2 \lambda T}-1}\Big) \mu({\rm d}y) \Big)^{-\frac{\bs-1}\bs}.
\end{align*}
Combined with the above estimate and the fact that  
\begin{align*}
\|p_T (x, \cdot)\|_{L^\bs(\mu)} 
=\sup_{\phi \in \mathcal G} \{\< p_T (x, \cdot), \phi\>_\mu \} 
=\sup_{\phi \in \mathcal G} \{P_T \phi(x) \}, 
\end{align*} 
we derive the density estimate \eqref{den}.
\end{proof}

\begin{rk}
According to \cite[Theorem 1.4.2 (1) and (2)]{Wan13}, the Harnack inequalities \eqref{har-log}-\eqref{har-pow} are equivalent to the following two heat kernel inequalities, respectively, provided $P_T$ has a strictly positive density $p_T(x, \cdot)$ concerning an invariant measure $\mu$:
\begin{align*}
\int_{L^p} p_T(x, z) \log \frac{p_T(x, z)}{p_T(y, z)} \mu({\rm d}z)
& \le \frac{\lambda \|G^{-1}\|_\infty^2 \|x-y\|_p^2}{e^{2 \lambda T}-1},  \\
\int_{L^p} p_T(x, z) \Big(\frac{p_T(x, z)}{p_T(y, z)}\Big)^\frac1{\bs-1} \mu({\rm d}z)
& \le \exp\Big(\frac{\bs \lambda\|G^{-1}\|_\infty^2 \|x-y\|_p^2}{(\bs-1)^2 (e^{2 \lambda T}-1)}\Big).  
\end{align*}
Under the conditions in Corollary \ref{cor-den}, $(P_t)$ has a unique invariant measure $\mu$ such that $p_T(x, \cdot)$ is strictly positive.
Then, the above two heat kernel inequalities are direct consequences of Theorem \ref{tm-har}.
\end{rk}

\section*{Declarations}

{\bf Conflict of interest}. The authors have no competing interests to declare relevant to this article's content.

\section*{Acknowledgements}

We thank the anonymous referees for their helpful comments and suggestions.
The author is supported by the National Natural Science Foundation of China, No. 12101296, Guangdong Basic and Applied Basic Research Foundation, No. 2024A1515012348, and Shenzhen Basic Research Special Project (Natural Science Foundation) Basic Research (General Project), Nos. JCYJ20220530112814033 and JCYJ20240813094919026.

\bibliographystyle{amsplain}
\bibliography{bib}

\providecommand{\bysame}{\leavevmode\hbox to3em{\hrulefill}\thinspace}
\providecommand{\MR}{\relax\ifhmode\unskip\space\fi MR }
\providecommand{\MRhref}[2]{%
  \href{http://www.ams.org/mathscinet-getitem?mr=#1}{#2}
}
\providecommand{\href}[2]{#2}
\begin{thebibliography}{10}

\bibitem{BGG12(JFA)}
F.~Bolley, I.~Gentil, and A.~Guillin, \emph{Convergence to equilibrium in
  {W}asserstein distance for {F}okker-{P}lanck equations}, J. Funct. Anal.
  \textbf{263} (2012), no.~8, 2430--2457. \MR{2964689}

\bibitem{Brz97(SSR)}
Z.~Brze\'zniak, \emph{On stochastic convolution in {B}anach spaces and
  applications}, Stochastics Stochastics Rep. \textbf{61} (1997), no.~3-4,
  245--295. \MR{1488138}

\bibitem{BG99(SPA)}
Z.~Brze\'{z}niak and D.~G\c{a}tarek, \emph{Martingale solutions and invariant
  measures for stochastic evolution equations in {B}anach spaces}, Stochastic
  Process. Appl. \textbf{84} (1999), no.~2, 187--225. \MR{1719282}

\bibitem{BK18(FS)}
Z.~Brze\'{z}niak and T.~Kok, \emph{Stochastic evolution equations in {B}anach
  spaces and applications to the {H}eath-{J}arrow-{M}orton-{M}usiela
  equations}, Finance Stoch. \textbf{22} (2018), no.~4, 959--1006. \MR{3860612}

\bibitem{BLS10(JEE)}
Z.~Brze\'{z}niak, H.~Long, and I.~Sim\~{a}o, \emph{Invariant measures for
  stochastic evolution equations in {M}-type 2 {B}anach spaces}, J. Evol. Equ.
  \textbf{10} (2010), no.~4, 785--810. \MR{2737159}

\bibitem{BR16(DCDS)}
Z.~Brze\'{z}niak and P.~Razafimandimby, \emph{Irreducibility and strong
  {F}eller property for stochastic evolution equations in {B}anach spaces},
  Discrete Contin. Dyn. Syst. Ser. B \textbf{21} (2016), no.~4, 1051--1077.
  \MR{3483552}

\bibitem{BS20}
O.~Butkovsky and M.~Scheutzow, \emph{Couplings via comparison principle and
  exponential ergodicity of {SPDE}s in the hypoelliptic setting}, Comm. Math.
  Phys. \textbf{379} (2020), no.~3, 1001--1034. \MR{4163359}

\bibitem{Cer03(PTRF)}
S.~Cerrai, \emph{Stochastic reaction-diffusion systems with multiplicative
  noise and non-{L}ipschitz reaction term}, Probab. Theory Related Fields
  \textbf{125} (2003), no.~2, 271--304. \MR{1961346}

\bibitem{Cer05(PTRF)}
\bysame, \emph{Stabilization by noise for a class of stochastic
  reaction-diffusion equations}, Probab. Theory Related Fields \textbf{133}
  (2005), no.~2, 190--214. \MR{2198698}

\bibitem{PZ96}
G.~Da~Prato and J.~Zabczyk, \emph{Ergodicity for infinite-dimensional systems},
  London Mathematical Society Lecture Note Series, vol. 229, Cambridge
  University Press, Cambridge, 1996. \MR{1417491}

\bibitem{PZ14}
\bysame, \emph{Stochastic equations in infinite dimensions}, second ed.,
  Encyclopedia of Mathematics and its Applications, vol. 152, Cambridge
  University Press, Cambridge, 2014. \MR{3236753}

\bibitem{NPV12}
E.~Di~Nezza, G.~Palatucci, and E.~Valdinoci, \emph{Hitchhiker's guide to the
  fractional {S}obolev spaces}, Bull. Sci. Math. \textbf{136} (2012), no.~5,
  521--573. \MR{2944369}

\bibitem{ET89}
P.~\'{E}rdi and J.~T\'{o}th, \emph{Mathematical models of chemical reactions},
  Nonlinear Science: Theory and Applications, Princeton University Press,
  Princeton, NJ, 1989, Theory and applications of deterministic and stochastic
  models. \MR{981593}

\bibitem{Fun16}
T.~Funaki, \emph{Lectures on random interfaces}, SpringerBriefs in Probability
  and Mathematical Statistics, Springer, Singapore, 2016.

\bibitem{Hai02(PTRF)}
M.~Hairer, \emph{Exponential mixing properties of stochastic {PDE}s through
  asymptotic coupling}, Probab. Theory Related Fields \textbf{124} (2002),
  no.~3, 345--380. \MR{1939651}

\bibitem{HHL19(JDE)}
J.~Hong, C.~Huang, and Z.~Liu, \emph{Optimal regularity of stochastic evolution
  equations in {M}-type 2 {B}anach spaces}, J. Differential Equations
  \textbf{267} (2019), no.~3, 1955--1971. \MR{3945622}

\bibitem{HLL20(SPL)}
W.~Hong, S.~Li, and W.~Liu, \emph{Asymptotic log-{H}arnack inequality and
  applications for {SPDE} with degenerate multiplicative noise}, Statist.
  Probab. Lett. \textbf{164} (2020), 108810, 8. \MR{4100888}

\bibitem{Kaw05(PA)}
H.~Kawabi, \emph{The parabolic {H}arnack inequality for the time dependent
  {G}inzburg-{L}andau type {SPDE} and its application}, Potential Anal.
  \textbf{22} (2005), no.~1, 61--84. \MR{2127731}

\bibitem{KN13(SPDE)}
S.~Kuksin and V.~Nersesyan, \emph{Stochastic {CGL} equations without linear
  dispersion in any space dimension}, Stoch. Partial Differ. Equ. Anal. Comput.
  \textbf{1} (2013), no.~3, 389--423. \MR{3327512}

\bibitem{Liu09(JEE)}
W.~Liu, \emph{Harnack inequality and applications for stochastic evolution
  equations with monotone drifts}, J. Evol. Equ. \textbf{9} (2009), no.~4,
  747--770. \MR{2563674}

\bibitem{Liu22}
Z.~Liu, \emph{{$L^p$}-convergence rate of backward {E}uler schemes for monotone
  {SDE}s}, BIT \textbf{62} (2022), no.~4, 1573--1590. \MR{4512616}

\bibitem{Liu25}
\bysame, \emph{Numerical ergodicity of stochastic {A}llen-{C}ahn equation
  driven by multiplicative white noise}, Commun. Math. Res. \textbf{41} (2025),
  no.~1, 30--44. \MR{4896563}

\bibitem{LL24}
Z.~Liu and Z.~Liu, \emph{Numerical {U}nique {E}rgodicity of {M}onotone {SDE}s
  {D}riven by {N}ondegenerate {M}ultiplicative {N}oise}, J. Sci. Comput.
  \textbf{103} (2025), no.~3, Paper No. 87. \MR{4899897}

\bibitem{LQ20(IMA)}
Z.~Liu and Z.~Qiao, \emph{Strong approximation of monotone stochastic partial
  differential equations driven by white noise}, IMA J. Numer. Anal.
  \textbf{40} (2020), no.~2, 1074--1093. \MR{4092279}

\bibitem{LQ21(SPDE)}
\bysame, \emph{Strong approximation of monotone stochastic partial differential
  equations driven by multiplicative noise}, Stoch. Partial Differ. Equ. Anal.
  Comput. \textbf{9} (2021), no.~3, 559--602. \MR{4297233}

\bibitem{LS25}
Z.~Liu and J.~Shen, \emph{Geometric ergodicity and optimal error estimates for
  a class of novel tamed schemes to super-linear {SPDE}s}, arXiv:2502.19117.

\bibitem{Nee01(OTAA)}
J.~M. A.~M. van Neerven, \emph{Uniqueness of invariant measures for the
  stochastic {C}auchy problem in {B}anach spaces}, Recent advances in operator
  theory and related topics ({S}zeged, 1999), Oper. Theory Adv. Appl., vol.
  127, Birkh\"{a}user, Basel, 2001, pp.~491--517. \MR{1902819}

\bibitem{NVW08(JFA)}
J.~M. A.~M. van Neerven, M.~C. Veraar, and L.~Weis, \emph{Stochastic evolution
  equations in {UMD} {B}anach spaces}, J. Funct. Anal. \textbf{255} (2008),
  no.~4, 940--993. \MR{2433958}

\bibitem{Wan97(PTRF)}
F.-Y. Wang, \emph{Logarithmic {S}obolev inequalities on noncompact {R}iemannian
  manifolds}, Probab. Theory Related Fields \textbf{109} (1997), no.~3,
  417--424. \MR{1481127}

\bibitem{Wan07(AOP)}
\bysame, \emph{Harnack inequality and applications for stochastic generalized
  porous media equations}, Ann. Probab. \textbf{35} (2007), no.~4, 1333--1350.
  \MR{2330974}

\bibitem{Wan10(JMPA)}
\bysame, \emph{Harnack inequalities on manifolds with boundary and
  applications}, J. Math. Pures Appl. (9) \textbf{94} (2010), no.~3, 304--321.
  \MR{2679029}

\bibitem{Wan11(AOP)}
\bysame, \emph{Harnack inequality for {SDE} with multiplicative noise and
  extension to {N}eumann semigroup on nonconvex manifolds}, Ann. Probab.
  \textbf{39} (2011), no.~4, 1449--1467. \MR{2857246}

\bibitem{Wan13}
\bysame, \emph{Harnack inequalities for stochastic partial differential
  equations}, SpringerBriefs in Mathematics, Springer, New York, 2013.
  \MR{3099948}

\bibitem{Wan17(JFA)}
\bysame, \emph{Hypercontractivity and applications for stochastic {H}amiltonian
  systems}, J. Funct. Anal. \textbf{272} (2017), no.~12, 5360--5383.
  \MR{3639531}

\bibitem{WZ14(SPA)}
F.-Y. Wang and T.S. Zhang, \emph{Log-{H}arnack inequality for mild solutions of
  {SPDE}s with multiplicative noise}, Stochastic Process. Appl. \textbf{124}
  (2014), no.~3, 1261--1274. \MR{3148013}

\bibitem{WZ13(JMPA)}
F.-Y. Wang and X.~Zhang, \emph{Derivative formula and applications for
  degenerate diffusion semigroups}, J. Math. Pures Appl. (9) \textbf{99}
  (2013), no.~6, 726--740. \MR{3055216}

\bibitem{Wan18(PTRF)}
Z.~Wang, \emph{A probabilistic {H}arnack inequality and strict positivity of
  stochastic partial differential equations}, Probab. Theory Related Fields
  \textbf{171} (2018), no.~3-4, 653--684. \MR{3827219}

\bibitem{Xie19(JDE)}
B.~Xie, \emph{Hypercontractivity for space--time white noise driven {SPDE}s
  with reflection}, J. Differential Equations \textbf{266} (2019), no.~9,
  5254--5277. \MR{3912749}

\end{thebibliography}

\end{document}